\documentclass[reqno]{amsart}

\usepackage{amsmath}
\usepackage{amsfonts}
\usepackage{amsthm}
\usepackage{amssymb}
\usepackage{graphicx}
\usepackage{graphics}
\usepackage{bm}
\usepackage{dsfont}
\usepackage{color}
\usepackage{enumerate}
\usepackage[font=footnotesize]{caption}
\usepackage[all]{xy}
\usepackage{xypic}
\allowdisplaybreaks

\numberwithin{equation}{section}

\newtheorem{thm}{Theorem}[section]
\newtheorem{cor}[thm]{Corollary}
\newtheorem{lem}[thm]{Lemma}
\newtheorem{prop}[thm]{Proposition}
\theoremstyle{definition}

\newtheorem{dfn}[thm]{Definition}
\newtheorem{rmk}[thm]{Remark}

\newcommand{\N}{\mathds{N}}

\newcommand{\Z}{\mathds{Z}}

\newcommand{\R}{\mathds{R}}

\newcommand{\diff}{\mathrm{d}}

\begin{document}

\title[Morse homology for strongly indefinite functionals on Banach spaces]{Morse homology for strongly indefinite functionals on Banach spaces}

\author[L. Asselle]{Luca Asselle}
\address{Ruhr-Universit\"at Bochum, Universit\"atsstra\ss e 150, 44801, Bochum, Germany}
\email{luca.asselle@rub.de}

\author[S. Cingolani]{Silvia Cingolani}
\address{Universit\'a degli Studi di Bari Aldo Moro, Via Orabona 4, 70125 Bari, Italy}
\email{silvia.cingolani@uniba.it}

\author[M. Starostka]{Maciej Starostka}
\address{Gda\'nsk University of Technology, Gabriela Narutowicza 11/12, 80233 Gda\'nsk, Poland}
\email{maciej.starostka@pg.edu.pl}

\date{\today}

\begin{abstract}
In this paper we lay the foundations for the Morse theoretical study of strongly indefinite functionals on Banach manifolds by developing the local theory for a specific model class that captures several key analytical features also arising in the variational formulations of geometric problems such as Dirac-harmonic maps.
As a corollary, we obtain existence results of solutions to certain systems of quasilinear elliptic problems involving the $p$-area functional. 
Abstracting from the concrete setting, we then formulate general conditions ensuring that Morse homology is well-defined for strongly indefinite functionals on a Banach space. 
\end{abstract}

\maketitle





\vspace{-5mm}

\section{Introduction}

Since their development in the early 1920's, Morse theoretical methods have had a prominent role in the variational study of differential equations. In its original formulation, 
Morse theory deals with smooth functions $f$ on a finite dimensional manifold $M$ and gives a lower bound on the number of critical points of such functions in terms of the 
topological complexity of $M$, provided that critical points are non-degenerate. The topological complexity of $M$ 
is here measured by the sum of the Betti numbers of $M$, and the proof is obtained by looking at the topology changes of sublevel sets $\{f\leq a\}$, $a \in \R$. 
Applying Morse theory to differential equations, however, requires extending the theory to infinite-dimensional manifolds. Indeed, solutions of several differential equations 
(for instance geodesics, periodic orbits of Hamiltonian systems, and solutions of certain classes of partial differential equations) can be viewed as critical points of functionals defined on loop or mapping spaces. 
This problem has attracted the attention of many mathematicians, starting with the work of Palais and Smale, and has led to the development of theories that still play a central role in many 
areas of contemporary mathematics, notably Floer theory. 

At present, unfortunately, no Floer-theoretical approach is available to treat certain partial differential equations of fundamental importance, such as the equation of harmonic or Dirac–harmonic maps. These can be seen as critical points of a functional $f$, whose natural domain of definition is (a certain Hilbert bundle over) the space of maps with Sobolev regularity $W^{1,2}$ from a closed surface into a Riemannian manifold $Q$. Unless in very few examples, such a space does not carry a manifold structure. A direct attempt to address this issue (currently being investigated by the first author in collaboration with Brilleslijper and to appear in a forthcoming paper)  is to exploit the equivalent Hamiltonian formulation for (Dirac-)harmonic maps within the framework of polysymplectic geometry, see \cite{Brill:2024}. However, this approach requires additional assumptions on  $Q$ (namely, a real analytic K\"ahler structure), since at present severe compactness issues can be overcome only within a hyperk\"ahler neighborhood of the zero-section in $T^*Q$. 
A complementary strategy, first introduced by Sacks and Uhlenbeck in \cite{Sacks:1981} for harmonic maps and then extended to Dirac-harmonic maps in \cite{Jost:2021}, is to replace $f$ with a functional $f_\alpha$ which is well-defined on a space of $W^{1,2\alpha}$-maps, for some $\alpha>1$. The gain in regularity yields a Banach manifold structure, giving a more suitable analytical setting for critical point theory. Using a minimax argument, the authors of \cite{Jost:2021,Sacks:1981} proved the existence of critical points for $f_\alpha$ and information on $f$ is then obtained by taking the limit $\alpha \downarrow 1$. However, unlike in the harmonic map case, the limiting procedure may produce Dirac-harmonic maps with trivial spinorial component. On the other hand,  the existence of genuine Dirac-harmonic maps is guaranteed under suitable assumptions, and for flat targets even Morse homology\footnote{The case of flat targets is substantially simpler, since one has a well-defined Hilbert structure and the functional is of the form ``compact perturbation of a fixed quadratic form''.} can be constructed, see \cite{Isobe:2025} and the references therein. This naturally suggests attempting to construct a Morse homology for $f_\alpha$, and then trying to pass the information to the limit as $\alpha\downarrow 1$ (at least in the case where $Q$ is negatively curved). Unfortunately, this turns out to be a difficult task. Indeed, the available literature for the construction of Morse homology in Banach manifolds, see  \cite{Abbondandolo:2006lk,Asselle:2025,Asselle:2024}, deals only with functionals having finite Morse index, and hence cannot be applied to strongly indefinite\footnote{This roughly speaking means that the Morse index and co-index of critical points is always infinite.} functionals as $f_\alpha$. Also, going beyond the finite Morse index case requires to address several issues which are typical of strongly indefinite functionals and which will be discussed below (the interested reader can find more details in \cite{AM:05,Asselle:2022} and references therein).

With this in mind, in the present paper we provide the foundations to the Morse theoretical study of strongly indefinite functionals on Banach manifolds by developing the local (linear) theory for a specific model class, which captures most of the key analytical features that also arise in the variational formulations of geometric problems such as Dirac-harmonic maps. Indeed, although not directly related to Dirac-harmonic maps, the functionals considered below in \eqref{eq:functional} reproduce precisely the analytical difficulties that prevent the application of standard Morse theoretical machinery: strong indefiniteness, the lack of a Hilbert setting, and the need for an additional compatible structure. 
Abstracting from the concrete setting, in Section 5 we will then provide general conditions for strongly indefinite functionals on a Banach space ensuring that Morse homology is well-defined. Extending the local theory to a global one will be subject of subsequent work. 

\vspace{2mm}

Thus, let $\Omega \subset \R^n$, $n\geq 2$, be a bounded domain with sufficiently regular boundary, and let $p, q>\frac n2$. We consider the $C^2$-smooth functional 
\begin{align}
& f: X:= W^{1,2p}_0(\Omega)\times W^{1,2q}_0(\Omega) \to \R, \nonumber \\ & f(u,v) := \frac 1{2p}  \int_{\Omega} \big (1+ |\nabla u|^2)^p\, \diff x - \frac 1{2q} \int_\Omega \big (1+|\nabla v|^2)^q\, \diff x - \int_\Omega G(u,v)\, \diff x,
\label{eq:functional}
\end{align}
where $G:\R^2 \to \R$ is a function of class $C^2$ such that
\begin{equation}
|G(x_1,x_2)| \lesssim 1 +  |x_1|^{\alpha_1} + |x_2|^{\alpha_2},  \quad \forall (x_1,x_2)\in \R^2,
\label{eq:growthG}
\end{equation}
with $0\leq \alpha_1<\beta_1(p,q)$, $0\leq \alpha_2< \beta_2(p,q)$. Here
$$\left \{\begin{array}{r} \beta_1(p,q)>p, \ \ \beta_2(p,q)>q \qquad \text{if} \ p\leq q < 2p-1\ \text{or}\ q\leq p <2q-1,\\  \beta_1(q,p)=\beta_2(q,p) := \min \{p,q\} \ \qquad \qquad \qquad \qquad \qquad \ \, \text{otherwise}.\end{array}\right.$$
so that linear growth conditions for $G$ (namely $\alpha_1=p$ and $\alpha_2=q$) are allowed if $p\leq q<2p-1$ (and, similarly, if $q\leq p < 2q-1$).
Condition \eqref{eq:growthG} and the assumption $p,q>\frac n2$ are not optimal and can likely be weakened. 
Extending the construction to the regime $p,q\leq \frac n2$, however, would require combining the foundational arguments developed here 
with the finer analysis of \cite{Asselle:2025}, to address the possible lack of persistence of the splitting 
 induced by the second differential at critical points. Since the aim of this paper is to present the basic, conceptual aspects of the construction of Morse homology rather than to seek minimal hypotheses, we postpone these extensions to future work, where we will also allow for more general principal parts in the spirit of \cite{Asselle:2025}. 
 Finally, unlike the ``one equation case'' considered in \cite{Asselle:2025,Asselle:2024}, we do not expect the construction to extend to superlinear growth conditions in general. Indeed, as shown in \cite{Liu:2010} (see also \cite[Lemma 4.3]{Asselle:2025}), in the one equation case superlinear growth conditions for the non-linearity require  a ``linear'' control from below in order to obtain the Palais-Smale (or Cerami) condition. In the present framework this would correspond to a two-sided linear control.

Hereafter we use the symbol ``$\lesssim$'' for any inequality which holds up to a multiplicative constant. 
As it is well known, critical points of the functional $f$ in \eqref{eq:functional} correspond to non-trivial solutions of the following system of quasilinear elliptic problems 
\begin{equation}
\left \{\begin{array}{r}
 \text{div} \Big [\big (1+|\nabla u|^2 \big )^{p-1} \nabla u \Big ] = D_u G(u,v) \quad \text{in} \ \Omega, \\ 
- \text{div} \Big [\big (1+|\nabla v|^2 \big )^{q-1} \nabla v \Big ] = D_v G(u,v) \quad \text{in} \ \Omega, \\
u=v=0 \qquad \qquad  \qquad \qquad \text{on} \ \partial \Omega. \end{array}\right .
\label{eq:system}
\end{equation}

Systems involving quasilinear operators of $p$-laplacian type model several phenomena in non-Newtonian mechanics, nonlinear elasticity and glaciology, combustion theory, population biology; see for instance \cite{Diaz:1994,Glowing:2003,Manasevich:1998}. Existence, non-existence and regularity results for such quasilinear elliptic systems are obtained by various authors, see e.g.  \cite{Boccardo:2002,Bozhkov:2003,Velin:1993}. 
As far as the system in \eqref{eq:system} is concerned, standard regularity theory shows that solutions of it are actually contained in $C^1(\overline \Omega)$, see \cite{Lieberman:1988}, and their existence
 follows immediately once we prove that, for a generic $f$, Morse homology is well-defined and isomorphic to the singular homology of $X$. 
To the best of our knowledge, systems of the form \eqref{eq:system} had not previously been analyzed through Morse-theoretical methods, due to their strong indefiniteness and the Banach nature of the ambient space.
In fact, systems such as \eqref{eq:system} are of a completely different nature than those considered in \cite{Carmona:2013}, unless $G$ is of the form $g(u)+h(v)$, in which case the system actually reduces to two uncoupled elliptic equations, and we shall see later that the Morse index and co-index of solutions to \eqref{eq:system}, regarded as critical points of $f$ in \eqref{eq:functional}, is always infinite, and in fact, \eqref{eq:system} may be viewed as  a Banach-analogue of the indefinite elliptic systems involving the Laplacian considered in \cite{Angenent:1999}. For this reason, the classical approach to Morse theory based on deformation of sublevel sets and cell attachments and used for instance in \cite{Carmona:2013,Cingolani:2003} is of no use here, as critical groups always vanish\footnote{One usually says that critical points with infinite Morse index and co-index are invisible to homotopy.}. 
We also remark that the existence of solutions of \eqref{eq:system} could likely be obtained also by other methods, for instance via finite dimensional approximations or through the Conley index approach discussed in \cite{Izy}. However, the Morse theoretical approach developed in the present paper is the most natural one for applications to non-linear settings (such as the case of Dirac-harmonic maps discussed above). 

In order to construct Morse homology for the functionals in \eqref{eq:functional} we need to address several issues. First, we must show that critical points are generically non-degenerate in the sense described below, which is the most natural one for the construction of Morse homology in a Banach (non Hilbert) setting (see \cite{Asselle:2025,Asselle:2024} for further details on this notion of non-degeneracy). 
\begin{dfn}
Let $Y$ be a Banach space, and let $F:Y\to \R$ be a functional of class $C^1$. A critical point $y_0$ of $F$ is called \textit{non-degenerate} if there exist a neighborhood $\mathcal U\subset Y$ of $y_0$ and a bounded linear hyperbolic operator $L:Y\to Y$ such that $F$ is a Lyapounov function on $\mathcal U$ for the linear flow induced by $L$. We call $F$ a \textit{Morse function} if all its critical points are non-degenerate. 
\label{def:nondegintro}
\end{dfn}

Non-degeneracy will be proved in Section \ref{sec:2} under the assumption that the second differential $\mathrm d^2 f:X\to X^*$ at each critical point is injective - a generic condition, as one sees by adapting the argument in \cite{Cingolani:2007}). The proof follows the line of \cite{Asselle:2025,Asselle:2024}, although some additional care is needed here because the Morse index and co-index of critical points of $f$ are always infinite. For this reason we also introduce a natural notion of relative Morse index, which is then used in the definition of the chain complex. 

Second, in a Banach setting there is no canonical way to associate a gradient flow to $f$. Therefore, in Section \ref{sec:3}, we construct a gradient-like vector field $W$ for $f$ such that the pair $(f,W)$ satisfies all the properties needed for the definition of Morse homology, beginning with the Palais-Smale condition. The latter will follow from the growth condition \eqref{eq:growthG}, which will also allow us to construct $W$ with at most linear growth and thus inducing a global flow. Such a linear control at infinity is, at present, essential for our approach: unlike in the Hilbert case, where normalizing a pseudo-gradient is typically harmless, in the Banach framework of this paper it is not clear to us how to ensure that the Palais–Smale condition is preserved after the normalization needed to produce a global flow.

Third, despite its crucial role, in the strongly indefinite case the Palais-Smale condition alone is unfortunately not enough to ensure that the intersection between stable and unstable manifolds of pairs of critical points is finite dimensional and pre-compact  (see for instance \cite{AM:05,Asselle:2022}). These two properties are then obtained by choosing $W$ within the smaller class of vector fields which are of the form ``compact perturbation of a fixed linear vector field'', hence in particular compatible with an additional structure naturally suggested by the problem (and which is need to make comparisons). 
Transversality is achieved then in Section \ref{sec:4} by a generic perturbation of the gradient-like vector field $W$ which does not destroy compactness. 

With all these issues being addressed, the construction of Morse homology is now standard. We define a chain complex generated by critical points and graded by the relative Morse index, whose boundary operator counts the number of flow lines (modulo two) connecting pairs of critical points whose relative Morse indices differ by one. The functoriality theorem (see Section \ref{sec:5}) shows that the resulting Morse homology is independent of the choice of $f$ and the gradient-like vector field $W$, thus yielding the following 

\begin{thm}
Let $f$ be as in \eqref{eq:functional}, with $G$ satisfying \eqref{eq:growthG}. Then, for a generic choice of $G$, the Morse homology $HM_*(f;\Z_2)$ is well-defined and isomorphic to $H_*(X,\Z_2)$. As a corollary, the system of quasilinear elliptic problems \eqref{eq:system} has at least one solution for every choice of non-linearity satisfying \eqref{eq:growthG}.
\label{thm:main}
\end{thm}

\vspace{2mm}

\textbf{Acknowledgments.} L.A. is partially supported by the DFG-grant 540462524 {\sl \lq\lq Morse theoretical methods in Analysis, Dynamics and Geometry"}. S.C. is supported by PNRR MUR project PE0000023 NQSTI - National Quantum Science and Technology Institute (CUP H93C22000670006), and partially supported by INdAM-GNAMPA.   M.S. is partially supported by the NCN grant 2023/05/Y/ST1/00186 {\sl \lq\lq Morse theoretical methods in Analysis, Dynamics and Geometry''}. 


\section{Preliminary computations, relative Morse index, and non-degeneracy}
\label{sec:2}

We readily compute for the differential of $f$ in  \eqref{eq:functional} in the $u$ resp. $v$ direction 
\begin{align}
\diff_u f(u,v)[(\xi,0)] &= \int_\Omega \big (1+|\nabla u|^2\big )^{p-1} \langle \nabla u,\nabla \xi \rangle \, \diff x - \int_\Omega \partial_1 G(u,v) \xi\, \diff x, \nonumber \\
\diff_v f(u,v)[(0,\eta)] &= - \int_\Omega \big (1+|\nabla v|^2\big )^{q-1} \langle \nabla v,\nabla \eta \rangle \, \diff x - \int_\Omega \partial_2 G(u,v) \eta\, \diff x.
\label{eq:differentialf}
\end{align}
If $(\bar u,\bar v) \in X$ is a critical point of $f$, then the second differential of $f$ at $(\bar u,\bar v) $ is 
\begin{align*}
\diff^2 & f(\bar u,\bar v) [(\xi_1,\eta_1),(\xi_2,\eta_2)] \\
& = \int_\Omega \big (1+|\nabla \bar u|^2\big )^{p-1} \langle \nabla \xi_1,\nabla \xi_2 \rangle \, \diff x + 2(p-1) \int_\Omega \big (1+|\nabla \bar u|^2)^{p-2} \langle \nabla \bar u,\nabla \xi_1\rangle \langle \nabla \bar u,\nabla \xi_2\rangle \, \diff x\\
& -  \int_\Omega \big (1+|\nabla \bar v|^2\big )^{q-1} \langle \nabla \eta_1,\nabla \eta_2 \rangle \, \diff x - 2(q-1) \int_\Omega \big (1+|\nabla \bar v|^2)^{q-2} \langle \nabla \bar v,\nabla \eta_1\rangle \langle \nabla \bar v,\nabla \eta_2\rangle \, \diff x\\
& - \int_\Omega \Big (\partial_{11} G(\bar u,\bar v) \xi_1\xi_2 + \partial_{12} G(\bar u ,\bar v ) (\xi_1\eta_2 + \xi_2\eta_1)+ \partial_{22}G(\bar u, \bar v) \eta_1\eta_2 \Big ) \, \diff x.
\end{align*}

Throughout the paper we assume that $$\diff^2 f(\bar u,\bar v):X\to X^*$$ is injective. 
Proceeding as in \cite{Asselle:2024}, and recalling that by standard regularity theory $(\bar u,\bar v) \in C^1(\overline \Omega)$, we introduce on $C^\infty_c (\Omega) \times C^\infty_c(\Omega)$ the scalar product 
\begin{align*}
\langle (\xi_1,\eta_1)&,(\xi_2,\eta_2) \rangle_{(\bar u, \bar v)} \\
			& := \int_\Omega \big (1+|\nabla \bar u|^2\big )^{p-1} \langle \nabla \xi_1,\nabla \xi_2 \rangle \, \diff x + 2(p-1) \int_\Omega \big (1+|\nabla \bar u|^2)^{p-2} \langle \nabla \bar u,\nabla \xi_1\rangle \langle \nabla \bar u,\nabla \xi_2\rangle \, \diff x\\
			& \ + \int_\Omega \big (1+|\nabla \bar v|^2\big )^{q-1} \langle \nabla \eta_1,\nabla \eta_2 \rangle \, \diff x + 2(q-1) \int_\Omega \big (1+|\nabla \bar v|^2)^{q-2} \langle \nabla \bar v,\nabla \eta_1\rangle \langle \nabla \bar v,\nabla \eta_2\rangle \, \diff x.
\end{align*}
and define the Hilbert space 
$$\mathbb H_{(\bar u,\bar v)} := \overline{C^\infty_c(\Omega) \times C^\infty_c (\Omega)}^{\langle \cdot,\cdot \rangle_{(\bar u, \bar v)}}.$$
Clearly, $\mathbb H_{(\bar u,\bar v)}\cong W^{1,2}_0(\Omega) \times W^{1,2}_0(\Omega)$. In particular, we have an induced continuous embedding $X \hookrightarrow \mathbb H_{(\bar u, \bar v)}$. 
The operator $\diff^2 f(\bar u,\bar v)$ extends to a bounded operator $H_{(\bar u,\bar v)}: \mathbb H_{(\bar u, \bar v)}\to \mathbb H_{(\bar u, \bar v)}^*$, which after using Riesz' representation theorem reads 
$$H_{(\bar u,\bar v)} = \left (\begin{matrix} \text{id} & 0 \\ 0 & -\text{id} \end{matrix}\right ) + K.$$
Here, $K$ is the compact operator uniquely determined by 
$$\langle K(\xi_1,\eta_1), (\xi_2,\eta_2)\rangle = -\int_\Omega \Big (\partial_{11} G(\bar u,\bar v) \xi_1\xi_2 + \partial_{12} G(\bar u ,\bar v ) (\xi_1\eta_2 + \xi_2\eta_1)+ \partial_{22}G(\bar u, \bar v) \eta_1\eta_2 \Big ) \, \diff x.$$
It follows that $H_{(\bar u,\bar v)}$ is a self-adjoint index-zero Fredholm operator whose spectrum consists of real eigenvalues with finite multiplicity accumulating to $\pm 1$. Accordingly, we have a $\langle \cdot,\cdot \rangle_{(\bar u,\bar v)}$-orthogonal decomposition 
$$\mathbb H_{(\bar u, \bar v)} \cong \mathbb H^- \oplus \mathbb H^0 \oplus \mathbb H^+$$
and there exists a constant $c>0$ such that 
\begin{align}
\langle H_{(\bar u,\bar v)} (\xi,\eta), (\xi,\eta)\rangle_{(\bar u,\bar v)} &\geq c \|(\xi,\eta)\|_{(\bar u,\bar v)}^2, \quad \forall (\xi,\eta) \in \mathbb H^+, \label{eq:pospart}\\
\langle H_{(\bar u,\bar v)} (\xi,\eta), (\xi,\eta)\rangle_{(\bar u,\bar v)} &\leq -c \|(\xi,\eta)\|_{(\bar u,\bar v)}^2, \quad \forall (\xi,\eta) \in \mathbb H^-. \label{eq:negpart}
\end{align}
Since $\dim \mathbb H^0 < +\infty$, standard regularity theory implies that $\mathbb H^0\subset X$, and this forces $\mathbb H^0=\{0\}$ because by assumption $\diff^2f(\bar u,\bar v)$ is injective. Thus, setting $X^\pm := X\cap \mathbb H^\pm$, we obtain the splitting 
$$X \cong X^- \oplus X^+,$$
with \eqref{eq:pospart} and \eqref{eq:negpart} holding on $X^+$ and $X^-$ respectively. This is precisely the kind of splitting obtained in \cite{Asselle:2024} (see also \cite{Asselle:2025,Cingolani:2003}). However, the crucial difference is that here $\dim X^- =+\infty$, namely the Morse index is infinite. To overcome this difficulty we replace the Morse index with a relative version of it, following \cite{AM:01}. 

We observe that $V:=\mathbb H^-$ is a compact perturbation\footnote{In the reference \cite{AM:01} the authors talk about commensurable subspaces. However,  the usual terminology in subsequent literature, e.g. \cite{AM:05,Asselle:2022}, is that of compact perturbations.} of $W:=\{0\}\times W^{1,2}_0(\Omega)$ meaning that the difference $P_V-P_W$ of the corresponding orthogonal projections  is a compact operator. This follows immediately from Proposition 2.2 in \cite{AM:01}, noticing that the difference $H_{(\bar u,\bar v)} - (\text{id},-\text{id}) = K$ is a compact operator. The identities 
\begin{align*}
P_{V^\perp}P_W &= (P_W-P_V)P_W, \qquad P_{W^\perp}P_V = (P_V-P_W)P_V,
\end{align*}
show then that both $P_{W^\perp}P_V$ and $P_{V^\perp} P_W$ are compact, hence the corresponding fixed points spaces, that is the subspaces $W\cap V^\perp$ and $W^\perp \cap V$, are finite dimensional.  
In particular, the relative dimension 
\begin{align*}
\dim (\{0\}\times W^{1,2}_0(\Omega), \mathbb H^- ) & := \dim ((\{0\}\times W^{1,2}_0(\Omega))^\perp \cap \mathbb H^- ) - \dim ((\{0\}\times W^{1,2}_0(\Omega))\cap (\mathbb H^-)^\perp)\\
			 & = \dim ((W^{1,2}_0(\Omega)\times \{0\}) \cap \mathbb H^- ) - \dim ((\{0\}\times W^{1,2}_0(\Omega))\cap \mathbb H^+)
\end{align*}
is a well-defined integer. We also notice that standard regularity theory implies that 
$$(W^{1,2}_0(\Omega)\times \{0\}) \cap \mathbb H^-\subset X^-,\quad (\{0\}\times W^{1,2}_0(\Omega))\cap \mathbb H^+\subset X^+,$$
so that 
$$\dim (\{0\}\times W^{1,2}_0(\Omega), \mathbb H^- ) = \dim ((W^{1,2p}_0(\Omega)\times \{0\}) \cap X^- ) - \dim ((\{0\}\times W^{1,2q}_0(\Omega))\cap  X^+).$$
\begin{dfn}
Let $(\bar u,\bar v)$ be a critical point of $f$ as in \eqref{eq:functional}. Then its \textit{relative Morse index} is defined as 
$$\mu_-(\bar u,\bar v) := \dim ((W^{1,2p}_0(\Omega)\times \{0\}) \cap X^- ) - \dim ((\{0\}\times W^{1,2q}_0(\Omega))\cap  X^+) \in \mathbb Z.$$
\end{dfn}

In the following lemma we show that the splitting $X\cong X^-\oplus X^+$ persists in a neighborhood of $(\bar u,\bar v)$. The proof follows the lines of \cite[Lemma 4.4]{Cingolani:2003} (see also \cite[Lemma 2.4]{Asselle:2024}), it requires however some extra care because of the strong indefinite nature of the functional $f$. 

\begin{lem}
There exists $r>0$ and $c'\in (0,c]$ such that for all $(u,v)\in B_r^X(\bar u,\bar v)$ we have
\begin{align*}
\diff^2 f(u,v)[(\xi,\eta),(\xi,\eta)] &\geq c' \|(\xi,\eta)\|_{(\bar u,\bar v)}^2 , \quad \forall (\xi,\eta)\in X^+,\\
\diff^2 f(u,v)[(\xi,\eta),(\xi,\eta)] &\leq -c' \|(\xi,\eta)\|_{(\bar u,\bar v)}^2 , \quad \forall (\xi,\eta)\in X^-.
\end{align*}
\label{lem:splittingpersistence}
\end{lem}

\vspace{-7mm}

\begin{proof}
We prove the first inequality, the proof of the second one being analogous. Assume by contradiction that there exists a sequence $(u_n,v_n)\subset X$ converging strongly in $X$ to $(\bar u,\bar v)$ and a sequence $(\xi_n,\eta_n)\subset X^+$ with $\|(\xi_n,\eta_n)\|_{(\bar u,\bar v)} =1$ for all $n\in \N$ such that 
\begin{equation}
\liminf_{n\to +\infty} \diff^2 f(u_n,v_n) [(\xi_n,\eta_n),(\xi_n,\eta_n)] \leq 0.
\label{eq:contradiction}
\end{equation}
Without loss of generality we can assume that $(\xi_n,\eta_n)$ converges weakly to $(\xi_\infty,\eta_\infty) \in \mathbb H^+$, and hence in particular strongly in $L^2$. 
Now, we compute 
\begin{align*}
\diff^2 f(u_n,v_n) & [(\xi_n,\eta_n),(\xi_n,\eta_n)] \\
			&= \underbrace{\int_\Omega (1+|\nabla u_n|^2 )^{p-1} |\nabla \xi_n|^2 + 2 (p-1) \int_\Omega (1+|\nabla u_n|^2)^{p-2} \langle \nabla u_n,\nabla \xi_n\rangle^2}_{=:(1)}  \\
			&- \underbrace{\int_\Omega (1+|\nabla v_n|^2 )^{q-1} |\nabla \eta_n|^2 - 2 (q-1) \int_\Omega (1+|\nabla v_n|^2)^{q-2} \langle \nabla v_n,\nabla \eta_n\rangle^2}_{=:(2)} \\ 
			&+ \underbrace{\int_\Omega \Big ( \partial_{11} G(u_n,v_n)\xi_n^2 + 2 \partial_{12} G(u_n,v_n) \xi_n\eta_n + \partial_{22} G(u_n,v_n) \eta_n^2\Big ).}_{=:(3)}
\end{align*}
As far as $(1)$ is concerned, arguing as in \cite[Lemma 2.4]{Asselle:2024} we see that the results in \cite{Ioffe:1977} imply that 
$$\liminf_{n\to +\infty} (1) \geq \int_\Omega (1+|\nabla \bar u|^2)^{p-1} |\nabla \xi_\infty|^2 + 2 (p-1) \int_\Omega (1+|\nabla \bar u|^2)^{p-2} \langle \nabla \bar u, \nabla \xi_\infty\rangle^2,$$
whereas for $(3)$ the embeddings $W^{1,2p}_0(\Omega)\hookrightarrow L^\infty (\Omega),W^{1,2q}_0(\Omega) \hookrightarrow L^\infty(\Omega)$ and the $L^2$-strong convergence of $(\xi_n,\eta_n)$ imply that 
$$(3) \to \int_\Omega \Big ( \partial_{11} G(\bar u ,\bar v)\xi_\infty^2 + 2 \partial_{12} G(\bar u,\bar v) \xi_\infty\eta_\infty + \partial_{22} G(\bar u,\bar v) \eta_\infty^2\Big ).$$
To estimate $(2)$ we first notice that, being $(\xi_n,\eta_n)\subset \mathbb H^+$ bounded and $\mathbb H^+$ a compact perturbation of $W^{1,2}_0(\Omega)\times \{0\}$, we have 
$$(0,\eta_n) = (\text{pr}_{\mathbb H^+}- \text{pr}_{W^{1,2}_0(\Omega)\times \{0\}})(\xi_n,\eta_n) \to (\text{pr}_{\mathbb H^+}- \text{pr}_{W^{1,2}_0(\Omega)\times \{0\}})(\xi_\infty,\eta_\infty) = (0,\eta_\infty)$$
strongly in $W^{1,2}$. Furthermore, since $\dim ((\{0\}\times W^{1,2}_0(\Omega)) \cap \mathbb H^+) <\infty$, this space is actually contained in $\{0\}\times W^{1,2q}_0(\Omega)$ by elliptic regularity, and 
this implies that $(0,\eta_n) \to (0,\eta_\infty)$ strongly in $W^{1,2q}$. A straightforward application of H\"older's inequality now shows that 
$$(2) \to - \int_\Omega (1+|\nabla \bar v|^2 )^{q-1} |\nabla \eta_\infty|^2 - 2 (q-1) \int_\Omega (1+|\nabla \bar v|^2)^{q-2} \langle \nabla \bar v,\nabla \eta_\infty\rangle^2.$$
To see this, we consider the first piece of (2) (the argument being identical for the second one), set $g_n:= (1+|\nabla v_n|^2)$ and $\chi_n := \nabla \eta_n$, observe that $g_n \to \bar g$ in $L^q$ whereas $\chi_n \to \chi_\infty$ in $L^{2q}$, and compute using H\"older's inequality with $\alpha = \frac q{q-1}$ and $\beta=q$
$$\int_\Omega g_n^{q-1} |\eta_n|^2 \leq \left (\int_\Omega g_n^q\right )^{\frac{q-1}q} \left (\int_\Omega |\eta_n|^{2q}\right )^{\frac 1 q} = \|g_n\|_q^{q-1}\|\eta_n\|_{2q}^2.$$
The claim follows now from Lebesgue's dominated convergence theorem. Putting all the pieces together we obtain that 
$$\liminf_{n\to+\infty}\diff^2 f(u_n,v_n)[(\xi_n,\eta_n),(\xi_n,\eta_n)] \geq \langle H_{(\bar u,\bar v)}(\xi_\infty,\eta_\infty), (\xi_\infty,\eta_\infty)\rangle \geq c \|(\xi_\infty,\eta_\infty)\|_{(\bar u,\bar v)}^2$$
and this forces $(\xi_\infty,\eta_\infty)=(0,0)$, as otherwise we would get a contradiction with \eqref{eq:contradiction}. In particular, $\eta_n \to 0$ strongly in $W^{1,2q}$ and $\xi_n\to 0$ strongly in $L^2$. On the other hand 
$$\diff^2 f(u_n,v_n)[(\xi_n,\eta_n),(\xi_n,\eta_n)] \geq \|\nabla \xi_n\|^2 - (2)- (3) \geq c \|\xi_n\|_{W^{1,2}} - (2)-(3)$$
which is equivalent to 
$$\|\xi_n\|_{W^{1,2}} \lesssim \diff^2 f(u_n,v_n)[(\xi_n,\eta_n),(\xi_n,\eta_n)] + (2)+ (3).$$
From this we deduce that 
$$\liminf_{n\to +\infty} \|\xi_n\|_{W^{1,2}} \lesssim \liminf_{n\to +\infty} \diff^2 f(u_n,v_n)[(\xi_n,\eta_n),(\xi_n,\eta_n)] + (2)+ (3) =0,$$
and hence that there is a subsequence of $(\xi_n)$ converging strongly to $0$ in $W^{1,2}$. This is a contradiction to the  assumption $\|(\xi_n,\eta_n)\|_{(\bar u,\bar v)}=1$ for all $n\in \N$.
\end{proof}

We can now proceed to show that critical points are non-degenerate in the sense of Definition \ref{def:nondegintro}. The proof is identical to \cite[Proposition 2.3]{Asselle:2024} and reported here just for the sake of completeness. 

\begin{prop}
Let $(\bar u,\bar v) \in X$ be a critical point of $f$ as in \eqref{eq:functional}. Then, there exist a neighborhood $\mathcal U$ of $(\bar u,\bar v)$ in $X$ 
and a linear hyperbolic operator $L:T_{(\bar u,\bar v)} X \to T_{(\bar u,\bar v)} X$ such that, on $\mathcal U$, $f$ is a Lyapunov function for the linear flow defined by $L$.  
\label{prop:hyperbolic}
\end{prop}

\begin{proof}
On 
$$T_{(\bar u,\bar v)} X \cong X = X^- \oplus X^+$$
we define the hyperbolic operator $L=(\text{id}, -\text{id})$, that is 
$$L x := L(x_- + x_+) := x_- - x_+, \quad \forall x = x_- + x_+ \in X.$$
We claim
that there exists a neighborhood of $(\bar u,\bar v)$ on which $f$ is a Lyapunov function for the linear flow defined by $L$,
which is equivalent to say that there exists $\delta>0$ such that
$$\mathrm{d} f((\bar u,\bar v) + x) [Lx] < 0, \quad \forall x \in B_	\delta^X(0) \setminus \{0\}.$$ 
We have 
$$\mathrm d f((\bar u,\bar v) + x)[\cdot] =\int_0^1 \frac{\mathrm d}{\mathrm d s} \big (\mathrm d f((\bar u,\bar v) + sx)\big )[\cdot]\, \diff s = \int_0^1 \mathrm d^2 f((\bar u,\bar v) + sx)[\cdot,x]\, \mathrm d s.$$
We choose $r,c'>0$ as in the statement of Lemma \ref{lem:splittingpersistence} and compute for $x\in X$ with $\|x\|<r$: 
\begin{align*}
\mathrm d &f((\bar u,\bar v) + x)[Lx] \\ 
   &= \int_0^1 \mathrm d^2 f((\bar u,\bar v) + sx)[Lx,x]\, \mathrm d s\\ 
			    &= \int_0^1 \mathrm d^2 f((\bar u,\bar v) + sx)[x_- - x_+, x_- + x_+] \, \mathrm d s\\
			    &= \int_0^1\Big ( \mathrm d^2 f((\bar u,\bar v) +sx)[x_-, x_-] -  \mathrm d^2 f((\bar u,\bar v) +sx)[x_+,x_+ ] \\
			    & \leq - c' \|x_-\|_{(\bar u,\bar v)}^2 - c' \|x_+\|_{(\bar u,\bar v)}^2\\ 
			    & \leq  -  c' \|x\|_{(\bar u,\bar v)}^2,
\end{align*}
thus proving the claim.
\end{proof}

\begin{rmk}
Adapting \cite{Cingolani:2007} to the setting of the present paper, we see that, for a generic choice of $G$ satisfying \eqref{eq:growthG}, every critical point $(\bar u,\bar v)$ of $f$ as in \eqref{eq:functional} is such that $\mathrm d^2f (\bar u,\bar v):X\to X^*$ is injective. Proposition \ref{prop:hyperbolic} therefore implies that $f$ is a Morse function in the sense of Definition \ref{def:nondegintro} for a generic choice of $G$. 
\end{rmk}


\section{The Palais-Smale condition, essentially vertical sets, and precompactness}
\label{sec:3}

Having established that $f$ as in \eqref{eq:functional} is generically a Morse function, we can proceed to the construction of Morse homology. To that purpose, 
let us first recall the fundamental steps of the construction in the case of finite Morse indices (see \cite{Asselle:2025,Asselle:2024} for the details). Using $C^2$-smooth partitions of unity 
(which always exist in the cases considered in the present paper, see \cite{Fry:2002}) and non-degeneracy of critical points,
one constructs a gradient-like vector field $V$ for $f$ (which is actually a true pseudo-gradient outside a neighborhood of the critical points) such that the pair $(f,V)$ satisfies the Palais-Smale condition (or the weaker Cerami condition in certain cases). This implies that 
the intersection 
$$W^u(y_0,V)\cap W^s(y_1,V)$$ 
of the stable manifold of the critical point $y_1$ with the unstable manifold of the critical point $y_0$ is contained in a bounded region (c.f. \cite[Proposition 2.5]{Asselle:2025}) and pre-compact, being the Morse index (thus, the dimension of $W^u$) finite. After a generic perturbation of $V$, we can assume that the pair $(f,V)$ satisfies the Morse-Smale condition up to order two, meaning that the intersection between stable and unstable manifolds is transverse whenever the difference of Morse indices is less than or equal to two. Therefore, for any $y_0,y_1 \in \text{crit}(f)$ whose Morse indices differ at most by two, the intersection 
$W^u(y_0,V)\cap W^s(y_1,V)$
is a precompact manifold of dimension equal the difference of the Morse indices. Now, the construction of Morse homology is standard: one defines a chain complex generated by critical points whose boundary operator counts the number (modulo two) of flow lines between pairs of critical points whose Morse indices differ by one (see for instance \cite{Abbondandolo:2006lk}).

To be able to implement such a scheme in the strongly indefinite setting of the present paper we first have to address several questions. We start by showing that functionals as in \eqref{eq:functional} for which the growth condition \eqref{eq:growthG} for the non-linearity holds satisfy the Palais-Smale condition. In what follows, we will need $f$ to satisfy the Palais-Smale condition with respect to a suitably chosen gradient-like vector field. As we shall see, this will follow from the construction of the gradient-like vector field and Lemma \ref{lem:Cerami} below. 

\begin{lem}
Let  $f:X\to \R$ be as in \eqref{eq:functional}, with the non-linearity $G$ satisfying \eqref{eq:growthG} with $\beta(p,q):=\min\{p,q\}$. Then, $f$ satisfies the Palais-Smale condition: every sequence $(u_k,v_k)\subset X$ such that 
$$f(u_k,v_k)\to c, \quad \|\diff f(u_k,v_k)\|_* \to 0$$
admits a converging subsequence.
\label{lem:Cerami}
\end{lem}

\begin{proof}
Let $(u_k,v_k)\subset X$ be a Palais-Smale sequence. We show that $(u_k,v_k)$ is bounded in $X$. By assumption
\begin{align*}
o(1) & = \frac{1}{\|u_k \|_{W^{1,2p}}} \diff f(u_k,v_k)[(u_k,0)] =  \frac{1}{\|u_k \|_{W^{1,2p}}}\left ( \int_\Omega (1+|\nabla u_k|^2)^{p-1}|\nabla u_k|^2 + \int_\Omega \partial_1 G(u_k,v_k)u_k\right )
	\end{align*}
	and similarly 
	\begin{align*}
o(1) & =  \frac{1}{\|v_k \|_{W^{1,2q}}}\diff f(u_k,v_k)[(0,v_k)] =  \frac{1}{\|v_k \|_{W^{1,2q}}}\left (- \int_\Omega (1+|\nabla v_k|^2)^{q-1}|\nabla v_k|^2 + \int_\Omega \partial_2 G(u_k,v_k)v_k\right ).
	\end{align*}
	The growth condition \eqref{eq:growthG} implies that 
	\begin{align*}
	|\partial_1 G(u_k,v_k)u_k| &\lesssim (1+|u_k|^{\alpha_1} + |v_k|^{\alpha_2}) |u_k| \\
			&\lesssim |u_k| + |u_k|^{\alpha_1+1} + |v_k|^{\alpha_2}|u_k|\\
			&\lesssim |u_k|^{\alpha_1+1} + |u_k|^2 + |v_k|^{2\alpha_2}
			\end{align*}
			and similarly
			\begin{align*}
	|\partial_2 G(u_k,v_k)v_k| &\lesssim (1+|u_k|^{\alpha_1} + |v_k|^{\alpha_2}) |v_k| \\
			&\lesssim |v_k| + |u_k|^{\alpha_1}|v_k| + |v_k|^{\alpha_2+1}\\
			&\lesssim |u_k|^{2\alpha_1} + |v_k|^2+ |v_k|^{\alpha_2+1}.
			\end{align*} 
			Putting all the estimates together we see that 
			\begin{align*}
			\|\nabla u_k\|_{2p}^{2p} & \lesssim \|u_k\|_{W^{1,2p}} + \|u_k\|_{{\alpha_1+1}}^{\alpha_1+1} +\|u_k\|_2^2 +  \|v_k\|_{{2\alpha_2}}^{2\alpha_2},\\
			\|\nabla v_k\|_{2q}^{2q} & \lesssim \|v_k\|_{W^{1,2q}} + \|u_k\|_{{2\alpha_1}}^{2\alpha_1} +\|v_k\|_2^2 +  \|v_k\|_{{\alpha_2+1}}^{\alpha_2+1}.
			\end{align*}
Summing up the two inequalities and using Sobolev's embedding theorem and Poincar\'e inequality yield
\begin{equation}
\|\nabla u_k\|_{2p}^{2p}+\|\nabla v_k\|_{2q}^{2q} \lesssim \|\nabla u_k\|_{2p}^{\max \{2, \alpha_1+1,2\alpha_1\}}+ \|\nabla v_k\|^{\max \{2, \alpha_2+1,2\alpha_2\}}
\label{eq:boundednessPS}
\end{equation}
which implies the boundedness of $(u_k,v_k)$ because by assumption 
$$\max \{2, \alpha_1+1,2\alpha_1\} < 2p, \quad \max \{2, \alpha_2+1,2\alpha_2\}< 2q.$$ 

Once we know that Palais-Smale sequences are bounded, the rest of the proof follows exactly as in the case of finite Morse index. We report the argument for completeness. Since $(u_k,v_k)$ is bounded, up to a subsequence we can assume $(u_k,v_k) \rightharpoonup (u_\infty,v_\infty)$ for some $(u_\infty,v_\infty)\in X$, hence in particular $(u_k,v_k)\to (u_\infty,v_\infty)$ in $L^\infty$. Clearly, we can write componentwise
$$\diff f (u_k,v_k)= (\diff_u f(u_k,v_k) ,\diff_v f(u_k,v_k)) = (D_p (u_k)+ K_u(u_k,v_k), - D_q(v_k) + K_v(u_k,v_k)),$$ 
where 
\begin{align}
& D_p:W^{1,2p}_0(\Omega) \to ( W^{1,2p}_0(\Omega))^*, \quad D_p (u)[\cdot] := \int_\Omega (1+|\nabla u|^2)^{p -1} \langle \nabla u ,\nabla \cdot \rangle\, \diff x \label{eq:Dp}\\ 
&D_q:W^{1,2q}_0(\Omega) \to ( W^{1,2q}_0(\Omega))^*, \quad D_q (v)[\cdot] := \int_\Omega (1+|\nabla v|^2)^{q -1} \langle \nabla v,\nabla \cdot \rangle\, \diff x \label{eq:Dq}
\end{align}
are invertible non-linear operators with continuous inverses (see \cite[Appendix B]{Benci})and 
\begin{align}
& K_u:X \to (W^{1,2p}_0(\Omega))^*, \quad K_u (u,v) [\cdot] := \int_\Omega \partial_1 G(u,v) \cdot \, \diff x \label{eq:Ku}\\
& K_v:X \to (W^{1,2q}_0(\Omega))^*, \quad K_v (u,v) [\cdot] := \int_\Omega \partial_2 G(u,v) \cdot \, \diff x. \label{eq:Kv}
\end{align} 
We readily see that $K_u(u_k,v_k) \to K_u(u_\infty,v_\infty)$ resp. $K_v(u_k,v_k)\to K_v(u_\infty,v_\infty)$ in operator norm. Consequently, the fact that 
$$o(1) = \|\diff f(u_k,v_k)\| = \|(D_p(u_k),D_q(v_k)) + (K_u(u_k,v_k),K_v(u_k,v_k))\|$$ 
implies that $D_p(u_k) \to K_u(u_\infty,v_\infty)$ and $D_q (v_k) \to K_v(u_\infty,v_\infty)$, and hence finally $u_k \to D_p^{-1}(K_u(u_\infty,v_\infty))$ resp. $v_k\to D_q^{-1}(K_v(u_\infty,v_\infty))$ by continuity of $D_p^{-1}$ and $D_q^{-1}$. 
\end{proof}

\begin{rmk}
The estimates in the proof of Lemma \ref{lem:Cerami} are by no means optimal. Indeed, the Palais-Smale condition continues to hold even after replacing $\alpha_1,\alpha_2<\min\{p,q\}$ with the weaker requirement that $\alpha_1<\beta(p)$, $\alpha_2<\beta(q)$, for suitable $\beta(p)>p$ and $\beta(q)>q$. To see this, we apply Young's inequality with $s_j\geq 2$, $j=1,2$, and $\frac{1}{r_j}+\frac{1}{s_j}=1$, to obtain
$$|u_k|^{\alpha_1}|v_k|\lesssim |u_k|^{r_1 \alpha_1}+ |v_k|^{s_1}, \quad |v_k|^{\alpha_2}|u_k| \lesssim |v_k|^{r_2\alpha_2} + |u_k|^{s_2}.$$ 
Therefore, the exponents in \eqref{eq:boundednessPS} can be replaced by 
$$\max \{s_2, \alpha_1 +1, r_1\alpha_1\}, \quad \max \{s_1, \alpha_2+1, r_2\alpha_2\},$$
and the condition for the boundedness of Palais-Smale sequences becomes 
$$\max \{s_2, \alpha_1+1,r_1\alpha_1\} < 2p, \quad \max \{s_1, \alpha_2+1,r_2\alpha_2\}< 2q.$$ 
Choosing $s_2=2p-\epsilon$, $s_1=2q-\epsilon$ yield 
$$r_1\alpha_1= \frac{2q-\epsilon}{2q-\epsilon-1} \alpha_1 < 2p, \quad r_2\alpha_2 = \frac{2p-\epsilon}{2p-\epsilon-1}\alpha_2 < 2q,$$
which are clearly satisfied for $\alpha_1=p$ and $\alpha_2 =q$ provided $\epsilon>0$ is small enough. Since obviously 
$p+1<2p$ and $q+1<2q$, we obtain that the Palais-Smale condition still holds for $\alpha_1 =p$ and $\alpha_2=q$ (and also for larger exponents). For later purposes, see Claim 2 below, it will be important that $\alpha_1,\alpha_2\leq \min \{2p-1,2q-1\}$. Therefore, we see that linear growth conditions for $G$ are allowed provided 
$$p\leq q <  2p-1 \qquad \text{or}\qquad  q\leq p < 2q-1.$$
Here we do not attempt to find the optimal values of $\beta(p)$ and $\beta(q)$ ensuring both the Palais-Smale condition and Claim 2 below, and leave this refinement to the interested reader. 
\label{rmk:linear}
\end{rmk}

\begin{rmk}
We shall stress the fact that Inequality \eqref{eq:boundednessPS} holds up to a constant which depends on $p,q,\alpha_1,\alpha_2$, and $\Omega$ but is independent of the specific choice of the Palais-Smale sequence $(u_k,v_k)$. Consequently, there exists $R=R(p,q,\alpha_1,\alpha_2,\Omega)>0$ such that any Palais-Smale sequence for $f$ is contained $B_R^X(0)$. In particular, all critical points of $f$ are contained in $B_R^X(0)$, and hence there exists $\lambda \in \R$ such that all critical points of $f$ are contained in the strip $\{-\lambda \leq f \leq \lambda\}$. 
\label{rmk:emptysublevel}
\end{rmk}

In the strongly indefinite setting of the present paper, the Palais-Smale condition, despite being crucial for the definition of Morse homology, is unfortunately not enough to ensure that the intersection of stable and unstable manifolds of pairs of critical points is finite dimensional and pre-compact. Indeed, for an arbitrary gradient-like vector field $V$, the Palais-Smale condition (or, if needed, the weaker Cerami condition) only implies that such an intersection is contained in a bounded region, but it might a priori be infinite dimensional, and even if finite dimensional, might contain infinitely many flow lines with no cluster points besides the critical points. As in \cite{AM:05,Asselle:2022}, finite dimensionality and pre-compactness are achieved by choosing the gradient-like vector field $V$ within the smaller class of vector fields which are compatible with an additional structure, needed in order to make comparisons. In non-linear settings such as \cite{AM:05,Asselle:2022} such an additional structure is given by a so-called $(0)$-essential subbundle $\mathcal E$ of the tangent bundle and needs to meet several compatibility conditions with $f$ and the (pseudo-)gradient flow. Usually, the choice of $\mathcal E$ is suggested by the problem itself. In the (easier) linear setting of the present paper, $\mathcal E$ is simply given by the constant subbundle of $TX$ given by $\{0\}\times W^{1,2q}_0(\Omega)$. As already proved in the previous section, the first compatibility condition between $\mathcal E$ and $f$ is satisfied; indeed, for every $(\bar u,\bar v) \in \text{crit}(f)$ we have that $X^-$ is a compact perturbation\footnote{The notion of compact perturbation makes sense in a Hilbert setting only. However, we can still say that $X^-$ is a compact perturbation of $\mathcal E$, meaning that the Hilbert extension $\mathbb H^-$ is a compact perturbation of $\{0\}\times W^{1,2}_0(\Omega)$.} of $\mathcal E$. 

\begin{dfn}
A bounded subset $\mathcal A\subset X = W^{1,2p}_0(\Omega) \times W^{1,2q}_0(\Omega)$ is called \textit{essentially vertical}\footnote{The role of the orthogonal projection in \cite{AM:05,Asselle:2022} is here played by the projection onto the first factor.} if $\text{pr}_1 (\mathcal A) \subset W^{1,2p}_0(\Omega) $ is pre-compact. 
Here $\text{pr}_1$ denotes the projection onto the first factor.
\end{dfn}

Notice that the family 
$$\mathcal F := \{\mathcal A \subset X\ |\ \mathcal A \ \text{essentially vertical}\}$$
is an essentially vertical family in the sense of Definition 2.5 in \cite{Asselle:2022}. Indeed, the first two properties are trivially satisfied. Also, $\mathcal F$ is closed with respect to the Hausdorff distance as one can show by proving that, if $\mathcal B \subset X$ is contained in the $\epsilon$-neighborhood of some $\mathcal A_\epsilon \in \mathcal F$ for every $\epsilon >0$, then $\mathcal B$ is bounded and $\text{pr}_1 (\mathcal B)$ is totally bounded (hence precompact). The details are left to the reader. 

\begin{dfn}
Let $W$ be a locally Lipschitz continuous vector field on $X$ whose flow $\varphi_W$ is globally defined. We say that $W$ \textit{preserves essentially vertical sets} if the following holds: for every essentially vertical set $\mathcal A \subset X$ and for every $T\geq 0$ the set 
$$\varphi_W \big ([0,T]\times \mathcal A) \subset X$$ 
is essentially vertical. 
\end{dfn}

\begin{rmk}
Even if in the present paper we will only use vector fields of the form ``compact perturbation of a fixed linear vector field'',
we still find convenient to formulate the next results in terms of vector fields which preserve essentially vertical sets, as this 
makes the connection with the references \cite{AM:05,Asselle:2022} clearer, while at the same time bridging more naturally towards applications in non-linear Banach settings. 
\end{rmk}

The proof of Proposition 4.1 in \cite{Asselle:2022} goes through word by word showing the following.

\begin{prop}
Let $\mathfrak W$ be the space of Lipschitz continuous vector fields on $X$ of class $C^2$, having at most linear growth, and which preserve essentially vertical sets. Then, the following hold: 
\begin{enumerate}
\item $\mathfrak W$ contains all vector fields of the form 
\begin{equation}
W(u,v) = (-u,v) + K(u,v)
\label{eq:compactperturbation}
\end{equation}
with $K:X\to X$ compact, of class $C^2$, and having at most linear growth. 
\item  $\forall B\subset X$ bounded, $\mathfrak W$ is a module over the ring of bounded Lipschitz function of class $C^2$ on $B$.
\end{enumerate}
\label{prop:modulestructure}
\end{prop}

We proceed now to construct a $C^2$-smooth gradient-like vector field $W$ for $f$ as in \eqref{eq:functional} which preserves essentially vertical sets. We shall stress the fact that $W$ will have 
the form \eqref{eq:compactperturbation}. To do that, we first 
define for every $(u,v) \in X$
\begin{equation}
 V(u,v) := (-u - D_p^{-1} K_u (u,v), v - D_q^{-1}K_v(u,v)),
\label{eq:tildeV}
\end{equation}
where $D_p,D_q,K_u,K_v$ are  as in \eqref{eq:Dp}--\eqref{eq:Kv} respectively. Notice that $V$, despite being obviously of the form \eqref{eq:compactperturbation}, it is a priori only continuous, hence need not even have a well-defined flow. Also, it is not immediately clear why $V$ should be gradient-like for $f$. 

\begin{rmk}
The operators $D_p$ and $D_q$ can be thought of as the natural (non-linear) replacements of the Riesz' isomorphism in the Banach setting of the present paper. 
It is worth noticing that, if $D_p$ and $D_q$ were actually linear, then we would have 
\begin{align*}
(-u - D_p^{-1} K_u (u,v) , v - D_q^{-1}K_v(u,v)) & = (- D_p^{-1}(D_p(u) + K_u(u,v)), - D_q^{-1}(-D_q(v) + K_v(u,v))) \\
 &= (-D_p^{-1}(\mathrm d_uf(u,v)), -D_q^{-1}(\mathrm d_vf(u,v)))\\
 &= -D^{-1}(\mathrm df(u,v)),
 \end{align*}
 that is, $V$ would agree with the natural gradient-like vector field $V_{\text{nat}}:=-D^{-1}(\mathrm d f)$ for $f$. Clearly, this is not the case unless $p=q=1$, which is however excluded in the present paper. The choice of $V$ instead of $V_{\text{nat}}$ is motivated by the fact that $V_{\text{nat}}$ need not preserve essentially vertical sets. 
\end{rmk}

\textbf{Claim 1.} $V$ is a gradient-like vector field for $f$. 

\vspace{2mm}

\noindent To this purpose, let us compute 
\begin{align*}
    \mathrm d f(u,v) [V(u,v)] &= \mathrm d_u f(u,v)[-u-D_p^{-1}K_u(u,v)] + \mathrm d_v f(u,v)[v-D_q^{-1}K_v(u,v)]\\ 
    						 &= (D_p(u) + K_u(u,v))[-u-D_p^{-1}K_u(u,v)] + (-D_q(v) +K_v(u,v))[v-D_q^{-1}K_v(u,v)]\\
						 &= - (D_p(u) - D_p (-D_p^{-1}K_u(u,v))) [u - (- D_p^{-1}K_u(u,v))]\\
						 &\ \ \ \, -  (D_q(v) - D_q (D_q^{-1}K_v(u,v))) [v - D_q^{-1}K_v(u,v)]\\
						 &= - (D_p(u) - D_p(\xi))[u-\xi] - (D_q(v)-D_q(\mu))[v-\mu],
\end{align*}
where $\xi := -D_p^{-1}K_u(u,v)$ and $\mu:= D_q^{-1}K_v(u,v)$. The claim follows as $D_p$ and $D_q$ are strictly convex:
$$\mathrm d f(u,v) [\tilde V(u,v)] \leq 0,\quad \forall (u,v),$$ 
with equality iff $u=\xi$ and $v=\mu$, that is, iff
\begin{align*}
\mathrm d_u f(u,v) &= D_p(u)+ K_u(u,v) = D_p (\xi) + K_u(u,v) = D_p (-D_p^{-1}K_u(u,v)) + K_u(u,v)) = 0,\\
\mathrm d_v f(u,v) &= -D_q(v)+ K_v(u,v) = -D_q (\mu) + K_v(u,v) = -D_q (D_q^{-1}K_v(u,v)) + K_v(u,v)) = 0,
\end{align*}
namely iff $(u,v)\in \text{crit}(f)$. For later use, we also observe the following stronger estimates, which can be found for instance in \cite[Chapter 9]{Brezis} or \cite[Section 24]{Zeidler}: 
\begin{align}
(D_p(u)-D_p(\xi))[u-\xi] &\geq 4c \|u-\xi\|_{W^{1,2p}}^{2p}, \label{eq:primastima}\\
(D_p(u)-D_p(\xi))[u-\xi] &\geq 4c \|D_p(u)-D_p(\xi)\|_*^{p'}, \label{eq:secondastima}
\end{align}
for some $c>0$. Here $p':= \frac{2p}{2p-1}$ is the conjugate exponent of $2p$. Analogous estimates hold for $D_q$ as well. 

\vspace{2mm}

\textbf{Claim 2.} $V$ has linear growth. 

\vspace{2mm}

\noindent First observe that 
$$V(u,v) = (-u,v) - (D_p^{-1}K_u(u,v), D_q^{-1}K_v(u,v)),$$
and hence 
$$\|V(u,v)\| \leq \|(u,v)\| + \| (D_p^{-1}K_u(u,v), D_q^{-1}K_v(u,v))\|.$$
Therefore, all we have to show is that the last term on the rhs grows at most linearly. Without loss of generality we consider only the term $D_p^{-1}K_u(u,v)$ (being the estimate for the other term identical) and,
for $\xi \in W^{1,2p}_0(\Omega)$ with $\|\xi \|_{W^{1,2p}}=1$, estimate using \eqref{eq:growthG} and the embedding $W^{1,2p}_0(\Omega) \hookrightarrow L^\infty(\Omega)$: 
\begin{align*}
 \big |K_u(u,v)[\xi]\big | &= \left |\int_\Omega \partial_1 G(u,v) \xi \right | \\
 			&\leq \int_\Omega \big |\partial_1 G(u,v)| |\xi|\\
			&\leq \|\xi\|_\infty \int_\Omega \big |\partial_1 G(u,v)|\\ 
			&\lesssim \|\xi\|_{W^{1,2p}} \big ( |\Omega| + \|u\|_{\alpha_1}^{\alpha_1} + \|v\|_{\alpha_2}^{\alpha_2}\big ).
\end{align*}
This implies that 
\begin{equation}
\|K_u(u,v)\|_* \lesssim 1 + \|u\|_{\alpha_1}^{\alpha_1} + \|v\|_{\alpha_2}^{\alpha_2}.
\label{eq:boundKu}
\end{equation}
Combining \eqref{eq:boundKu} with the standard estimate 
$$\|D_p^{-1} \varphi \|_{W^{1,2p}} \lesssim 1 + \|\varphi\|_*^{\frac{1}{2p-1}}, \quad \forall \varphi \in (W^{1,2p}_0(\Omega))^*,$$
we obtain 
\begin{align*}
\|D_p^{-1}K_u(u,v) \|_{W^{1,2p}} & \lesssim 1 + \|K_u(u,v)\|_*^{\frac{1}{2p-1}} \\ 
& \lesssim 1 + \big (1 + \|u\|_{\alpha_1}^{\alpha_1} + \|v\|_{\alpha_2}^{\alpha_2}\big)^{\frac{1}{2p-1}}\\
& \lesssim 1 + \|u\|_{\alpha_1}^{\frac{\alpha_1}{2p-1}} + \|v\|_{\alpha_2}^{\frac{\alpha_2}{2p-1}}\\
&\lesssim 1 + \|(u,v)\|^\gamma 
\end{align*}
with $\gamma:={\frac{\max \{\alpha_1,\alpha_2\}}{2p-1}}<1$ by Assumption \eqref{eq:growthG}. The claim follows. 

\vspace{2mm}

\textbf{Claim 3.} $V$ can be made to a $C^2$-smooth gradient-like vector field $W$ having linear growth so that $W$, in a neighborhood of any critical point $(\bar u,\bar v)$, coincides with the linear vector field induced by the hyperbolic operator $L_{(\bar u,\bar v)}$ given in Proposition \ref{prop:hyperbolic}. 

\vspace{2mm}

For each critical point  $(\bar u,\bar v)$ of $f$ consider $r=r(\bar u, \bar v)>0$ such that  $f$ is a Lyapounov function on $B_r(\bar u,\bar v)$  for the linear flow induced by the linear hyperbolic operator $L_{(\bar u, \bar v)}$ given in Proposition \ref{prop:hyperbolic}. On $B_r(\bar u,\bar v)$ we therefore set $W_{(\bar u,\bar v)}(u,v) := L_{(\bar u,\bar v)} x,$ where $x:= (u-\bar u,v-\bar v)$. 
Denote by $U$ the union of all $B_r(\bar u,\bar v)$, $(\bar u, \bar v) \in \text{crit}(f)$.
Notice now that \eqref{eq:secondastima} implies that 
\begin{equation}
\mathrm d f(u,v) [V(u,v)] \leq- 4c \big (\|\mathrm d_u f(u,v)\|_*^{p'} + \|\mathrm d_v f(u,v) \|_*^{q'}\big ), \quad \forall X\setminus U, \label{eq:generalizedpg}
\end{equation}
where $p',q'<2$ denote the conjugate exponents to $2p$ and $2q$ respectively. 
By \eqref{eq:generalizedpg} and the continuity of $\mathrm df$ we find, for each fixed $(u_0,v_0)\in X\setminus U$, $r_0=r_0(u_0,v_0)>0$ such that 
$$\mathrm d f(u,v) [V(u_0,v_0)] \leq - 2c \big (\|\mathrm d_u f(u_0,v_0)\|_*^{p'} + \|\mathrm d_v f(u_0,v_0) \|_*^{q'}\big),\quad \forall (u,v) \in B_{r_0}(u_0,v_0),$$
and hence, up to shrinking $r_0>0$ further if necessary, 
\begin{equation}
\mathrm d f(u,v) [V(u_0,v_0)] \leq - c \big (\|\mathrm d_u f(u,v)\|_*^{p'} + \|\mathrm d_v f(u,v) \|_*^{q'}\big),\quad \forall (u,v) \in B_{r_0}(u_0,v_0)
\label{eq:estimatefreeze}
\end{equation}
We therefore define $W_{(u_0,v_0)}(u,v):= V(u_0,v_0)$ for all $(u,v) \in B_{r_0}(u_0,v_0)$. 
Without loss of generality we can assume $B_{r_0}{(u_0,v_0)}\cap \frac 12 U =\emptyset$. Now we consider the open covering of $X$ given by
$$\mathfrak U= \mathfrak U_{\text{crit}} \cup \mathfrak U_{\text{nc}} = \Big \{ B_r(\bar u,\bar v) \ \Big | \ (\bar u, \bar v) \in \text{crit}(f)\Big \} \cup \Big \{ B_{r_0}(u_0,v_0) \ \Big |\ (u_0,v_0) \in X\setminus U\Big \}.$$
By the paracompactness of $X$, there exists a locally finite refinement $\mathfrak V =\{\mathcal V_j \ |\ j\in J\}$ of the open covering $\mathfrak U$ which contains $\mathfrak U_{\text{crit}}$. 
Let $\Gamma : J \to X$ be a function such that $\mathcal V_j \subseteq B_{r(\Gamma(j))}(\Gamma(j))$ for all $j\in J$, with equality if $\Gamma(j) \in \text{crit}(f)$. Following \cite{Bonic:1966,Fry:2002}, the Banach space $X$ admits $C^2$-smooth bump functions. Therefore, we can find a $C^2$-smooth partition of unity $\{\chi_j\}_{j\in J}$ subordinated to $\mathfrak V$. By construction we have 
$$\Gamma(j) = (\bar u,\bar v) \in \text{crit}(f) \ \ \Rightarrow \ \ \chi_j \equiv 1 \quad \text{on} \ \frac 12 B_r(\bar u,\bar v).$$
Finally, we set $W:X\to X$ by 
$$W(u,v) := \sum_{j\in J} \chi_j(u,v) W_{\Gamma(j)}(u,v).$$
By construction $W$ satisfies all desired properties.

\vspace{2mm}

\textbf{Claim 4.} The pair $(f,W)$ satisfies the Palais-Smale condition, meaning that every sequence $(u_k,v_k)\subset X$ such that 
$f(u_k,v_k) \to c$ and  $\mathrm df(u_k,v_k)[W(u_k,v_k)] \to 0$
admits a converging subsequence.

\vspace{2mm}

If $(u_k,v_k)$ has a subsequence which converges to a critical point then there is nothing to prove. Therefore, we can assume without loss of generality that $(u_k,v_k) \subset X\setminus U$. 
Using \eqref{eq:estimatefreeze} we obtain that 
\begin{align*}
o(1) &= \mathrm d f(u_k,v_k) [W(u_k,v_k)]\\
	&= \sum_{j\in J} \chi_j(u_k,v_k)  d f(u_k,v_k) [W_{\Gamma(j)} (u_k,v_k)]\\
	&= \sum_{j\in J} \chi_j(u_k,v_k)  d f(u_k,v_k) [V(u_j,v_j)]\\
	&\leq - c \sum_{j \in J} \chi_j(u_k,v_k) \big (\|\mathrm d_u f(u_k,v_k)\|_*^{p'} + \|\mathrm d_v f(u_k,v_k) \|_*^{q'}\big)\\
	&= -c \big (\|\mathrm d_u f(u_k,v_k)\|_*^{p'} + \|\mathrm d_v f(u_k,v_k) \|_*^{q'}\big).
	\end{align*}
From this we deduce that $\|\mathrm d f(u_k,v_k)\|_*\to 0$. The claim follows then from Lemma \ref{lem:Cerami}.

\vspace{2mm}

\textbf{Claim 5.} $W$ preserves essentially vertical sets. 

\vspace{2mm}

Clearly, $W_{\Gamma(j)}$ preserves essentially vertical sets for every $j\in J$, and hence so does every $\chi_j\cdot W_{\Gamma(j)}$ by Proposition \ref{prop:modulestructure}. Using again Proposition \ref{prop:modulestructure} we obtain that 
$$\sum_{j \in J_\ell} \chi_j W_{\Gamma(j)} \in \mathfrak W, \quad \forall J_\ell \subset J \ \text{finite}.$$
Since the partition of unity $\{\chi_j\}$ is locally finite, and since $\mathcal F$ is closed with respect to the Hausdorff distance, an argument analogous to the one in \cite[Section 4]{Asselle:2022}, by passing to the limit $|J_\ell|\to +\infty$ we obtain that $W\in \mathfrak W$. 

\begin{rmk}
Claim 5 can also be shown by observing that $W$ is of the form $(-\text{id},\text{id}) + K$ with $K$ compact  (as every $W_{\Gamma(j)}$ does). The claim follows then from Proposition \ref{prop:modulestructure}.
\end{rmk}

Summarizing, we have proved the following

\begin{prop}
Let $f$ be as in \eqref{eq:functional} with non-linearity $G$ satisfying the growth condition \eqref{eq:growthG}. Then, there exists a $C^2$-smooth vector field $W:X\to X$ of the form 
$$W(u,v) = (-u,v) + K(u,v),\quad \forall (u,v) \in X,$$
such that the following hold: 
\begin{enumerate}
\item $W$ has linear growth, $K:X \to X$ is compact. 
\item $W$ is gradient-like for $f$.
\item $(f,W)$ satisfies the Palais-Smale condition. 
\item Every $(\bar u ,\bar v)\in \mathrm{crit}(f)$ has a neighborhood over which $W$ coincides with the linear vector field given in Proposition \ref{prop:hyperbolic}. 
\item $W$ preserves essentially vertical sets. 
\end{enumerate}
\label{prop:gradientc2}
\end{prop}

Adapting the proof of \cite[Theorem 6.5]{AM:05} to the setting of the present paper we obtain the following 

\begin{thm}
Let $f$ be as in \eqref{eq:functional} with $G$ satisfying \eqref{eq:growthG}. Assume that $f$ is Morse and that $V$ is a $C^2$-smooth gradient-like vector field for $f$ whose flow is globally defined and such that: 
\begin{enumerate}
\item[i)] $V(x) = Lx$ in  a neighborhood of every $(\bar u, \bar v) \in \mathrm{crit}(f)$, where $L$ is given by Proposition \ref{prop:hyperbolic}. 
\item[ii)] $(f,V)$ satisfies the Palais-Smale condition, and 
\item[iii)] $V$ preserves essentially vertical sets. 
\end{enumerate}
Then, for any $(u_-,v_-),(u_+,v_+) \in \mathrm{crit}(f)$, the intersection 
$$W^u((u_-,v_-);V) \cap W^s((u_+,v_+);V)$$ 
is pre-compact. In particular, if transverse and non-empty, it is a finite dimensional submanifold of $X$. 
\label{teo:precompactness}
\end{thm}

We refrain to give a complete proof of Theorem \ref{teo:precompactness} here because it follows from the proof of Theorem 6.5 in \cite{AM:05} with some minor adaptation, and just give an idea of the proof instead. Take a sequence $\{x_n\} \subset W^u((u_-,v_-);V) \cap W^s((u_+,v_+);V)$. 
Assume without loss of generality that no subsequence converges to $(u_-,v_-)$ or $(u_+,v_+)$. A minimality argument on $(u_-,v_-)$ and $(u_+,v_+)$ (recall that $\text{crit}(f)$ is a finite set) and the Palais-Smale condition imply that there exist bounded sequences $\{t_n\}\subset (-\infty,0]$
and $\{s_n\}\subset [0,+\infty)$ 
such that 
$$\Phi_{t_n}^V(x_n) \in \partial B_r(u_-,v_-),\quad \Phi_{s_n}^V(x_n) \in \partial B_r(u_+,v_+),$$ 
where $B_r(u_-,v_-)$ and $B_r(u_+,v_+)$ are chosen so that, over them, $V$ coincides with the linear vector fields given by Proposition \ref{prop:hyperbolic}. This implies in particular that 
$$\{\Phi_{t_n}^V(x_n)\}\subset X_{(u_-,v_-)}^-,\quad \{ \Phi_{s_n}^V(x_n)\} \subset X_{(u_+,v_+)}^+,$$
hence 
$\{\Phi_{t_n}^V(x_n)\}$ is essentially vertical and $\{ \Phi_{s_n}^V(x_n)\}$ is essentially horizontal, meaning that its projection to $\{0\}\times W^{1,2q}_0(\Omega)$ is pre-compact. Now, since $V$ preserves essentially vertical sets, we further have that 
$$\{ \Phi_{s_n}^V(x_n) \} \subset \Phi^V \big ([0,T] \times \{\Phi_{t_n}^V(x_n)\}\big ), \quad T:= \max \{t_n-s_n\ |\ n\in \N\} < +\infty,$$ 
is essentially vertical. This implies that $\{ \Phi_{s_n}^V(x_n) \}$ is pre-compact, being its projections on both factors pre-compact, and the pre-compactness of $\{x_n\}$ follows. 

\begin{rmk}
At this point we cannot yet say that the dimension of 
$$W^u((u_-,v_-);V) \cap W^s((u_+,v_+);V)$$ 
is given by $\mu_-(u_-,v_-)- \mu_- (u_+,v_+)$, that is, by the difference of the relative Morse indices. This will follow, for $V$ of the form \eqref{eq:compactperturbation} and for any $\gamma:\R\to X$ flow line of $V$ contained in $W^u((u_-,v_-);V) \cap W^s((u_+,v_+);V)$, from the Fredholm properties 
of the operator
$$\frac{\mathrm d}{\mathrm dt} - \mathrm d V(\gamma(t)) : C^1_0(\R,X) \to C^0_0(\R,X).$$
As we shall see below, this is enough for us, because  transversality will be achieved by perturbing the gradient-like vector field $W$ given in Proposition \ref{prop:gradientc2} within the class of vector fields of the form \eqref{eq:compactperturbation}. 
The claim should remain true for arbitrary $V$ preserving essentially vertical sets (see \cite[Theorem E]{AM:03} for an analogous statement in the Hilbert setting), however we do not need such a result here. 
\end{rmk}


\section{Fredholm theory and transversality}

\label{sec:4}

In this section we show that we can generically perturb the gradient-like vector field $W$ given in Proposition \eqref{prop:gradientc2} in such a way that the pair $(f,W)$ satisfies the Morse-Smale condition up to order two, meaning that stable and unstable manifold of pairs of critical points whose relative Morse indices differ by at most two intersect transversally. As a bi-product we will show the piece of information which was missing in Theorem \ref{teo:precompactness}, namely that the dimension is given precisely by the difference of relative Morse indices.  

As a first step, we set\footnote{In what follows $X$ could be replaced by any reflexive Banach space.}
\begin{align*}
C^1_0(\R,X) &:= \Big \{\gamma\in C^1(\R,X)\ \Big |\ \lim_{t\to \pm \infty} \gamma(t) = \lim_{t\to \pm \infty} \dot \gamma(t) = 0\Big \},\\
C^0_0(\R,X) &:= \Big \{\gamma\in C^0(\R,X)\ \Big |\ \lim_{t\to \pm \infty} \gamma(t) = 0\Big \}.
\end{align*}
and consider operators of the form 
\begin{equation}
D_A:= \frac{\mathrm d}{\mathrm d t} - A(\cdot): C^1_0(\R,X) \to C^0_0(\R,X).
\label{eq:DA}
\end{equation}
where $A(\cdot) := A_0 + K(\cdot)$, with $K:\R\to L_c(X)$ a continuous path of compact operators on $X$ such that  
$$K(t) \to K(\pm \infty), \quad \text{for} \ t \to \pm \infty.$$

In the next lemma we discuss the Fredholm properties of $D_A$ following \cite{Lat:2003} (see also \cite{Rabier} for the particular case $A(-\infty)=A(+\infty)$ and \cite[Theorem B]{AM:03} for an analogous statement in the case of Hilbert spaces).

\begin{lem}
The operator $D_A$ defined in \eqref{eq:DA} is Fredholm if and only if 
the asymptotic operators $A(\pm \infty):= A_0 + K(\pm \infty)$ are iperbolic\footnote{Recall that $T\in L(X)$ is called hyperbolic if its spectrum $\sigma(T)$ is disjoint from $i\R$.}. 
In this case, the pair $(V^-(A(-\infty)),V^+(A(+\infty))$ is a Fredholm pair, where 
$$X = V^-(A(\pm \infty)) \oplus V^+(A(\pm \infty))$$
denotes the $A(\pm \infty)$-invariant splitting of $X$ into closed subspaces given by the spectral decomposition respectively, and we have
$$\mathrm{ind}\, D_A = \mathrm{ind} \big (V^-(A(-\infty)),V^+(A(+\infty))\big ).$$
\label{lem:Fredholmness}
\end{lem}
\begin{proof}
Following \cite[Section 7]{Lat:2003} we define the piecewise constant path of bounded operators
$$C:\R\to L(X), \quad C(t) := \left \{ \begin{array}{l} A(-\infty) \quad \text{for} \ t < 0,\\ A(+\infty) \quad \text{for} \ t\geq 0,\end{array}\right. $$
and 
$$\tilde K:\R\to L_c(X), \quad \tilde K(t) = \left \{ \begin{array}{l} K(t) - K(-\infty) \quad \text{for} \ t < 0, \\ K(t)-K(+\infty) \quad \text{for}\ t \geq 0,\end{array}\right.$$
so that 
$$A(t) = C(t) + \tilde K(t).$$
This way we see that $A(t)$ satisfies all assumptions\footnote{There is a sign discrepance between our setting and \cite{Lat:2003}, because there operators of the form $-\frac{\mathrm d}{\mathrm d t} + A(t)$ are considered.} of Proposition 7.15 in \cite{Lat:2003}, besides Property (P1) for $\tilde K(\cdot)$ which fails at $t=0$. However, a careful inspection of the proof of Proposition 7.6 in \cite{Lat:2003} shows that all assertions continue to hold even if  $\tilde K(\cdot)$ has isolated discontinuities, because this is enough to prove the compactness of the operators $K_n$ appearing in that proof. We are therefore allowed to apply Proposition 7.15 in our setting to conclude that $D_A$ is Fredholm if and only if $A(\pm \infty)$ are hyperbolic. The Fredholmness condition for the pair $(\text{Im} \, P_{A_+},\text{ker} \, P_{A_-})$ in Proposition 7.15 is redundant here, as one can readily see that such a pair is precisely given by $(V^-(A(-\infty)),V^+(A(+\infty)))$, which is Fredholm provided $A(\pm \infty)$ are hyperbolic. This last fact follows from the very structure of the path $A(t)$, which implies that the difference  $P^-(A(+\infty))- P^-(A(-\infty))$ of the spectral projections (corresponding to the part of the spectrum contained in $\{Re (\lambda )<0\}$) is a compact operator. The details are left to the reader. 
\end{proof}

Specifying Lemma \ref{lem:Fredholmness} to the setting of Section 3 we obtain the following 

\begin{cor}
Let $f$ be as in \eqref{eq:functional} with $G$ satisfying \eqref{eq:growthG}. Assume that $f$ is Morse and that $V$ is a $C^2$-smooth gradient-like vector field for $f$ of the form \eqref{eq:compactperturbation} satisfying additionally Hypotheses i)-ii)\footnote{Hypothesis iii) is satisfied as well since $V$ is of the form \eqref{eq:compactperturbation}.} of Theorem \ref{teo:precompactness}. Then, the following holds: for every $(u_-,v_-),(u_+,v_+)\in \mathrm{crit} (f)$ and for every $\gamma:\R\to X$ flow line of $V$ contained in $W^u((u_-,v_-);V) \cap W^s((u_+,v_+);V)$, the operator 
\begin{equation}
D_{DV}:\frac{\mathrm d}{\mathrm dt} - \mathrm dV(\gamma(t)): C^1_0(\R,V) \to C^0_0(\R,X)
\label{eq:DDV}
\end{equation}
is Fredholm with Fredholm index given by 
$$\mathrm{ind}\, D_{DV}= \mathrm{ind} \big (V^-(\mathrm dV(u_-,v_-))),V^+(\mathrm dV(u_+,v_+))\big ) = \mu_-(u_-,v_-)-\mu_-(u_+,v_+),$$
and $W^u((u_-,v_-);V)$ and $W^s((u_+,v_+);V)$ have Fredholm intersection at $\gamma(t)$ for every $t\in \R$, with 
$$\mathrm{ind} \big (T_{\gamma(t)} W^u((u_-,v_-);V), T_{\gamma(t)} W^s((u_+,v_+);V)\big ) = \mu_-(u_-,v_-)-\mu_-(u_+,v_+).$$
\label{cor:Fredholmness}
\end{cor}
\begin{proof}
The Fredholmness of $D_{DV}$ follows immediately from Lemma \ref{lem:Fredholmness} noticing that 
$$\mathrm d V(\gamma(t)) = (-\text{id},\text{id}) + \mathrm dK(\gamma(t)),$$
with $t\mapsto \mathrm d K(\gamma(t))$ continuous path of compact operators and that 
$$\mathrm dV(\gamma(t)) \to L_{(u_\pm,v_\pm)} \quad \text{for} \ t \to \pm \infty\footnote{Actually, by construction we have $\mathrm dV(\gamma(t)) \equiv L_{(u_\pm,v_\pm)}$, for $\pm t >\!\!> 1.$}.$$
The last part of the statement follows from the fact that $V^\pm(\mathrm dV(u_\pm,v_\pm))= X^\pm_{(u_\pm,v_\pm)}$, where $X=X^-_{(u_\pm,v_\pm)}\oplus X^+_{(u_\pm,v_\pm)}$ is the splitting introduced in Section 2, and the definition of the relative Morse index. 
\end{proof}

\begin{cor}
Under the assumptions of Corollary \ref{cor:Fredholmness}, we have that, for all $t\in \R$
\begin{align*}
\ker D_{DV} &\cong T_{\gamma(t)} W^u((u_-,v_-);V)\cap T_{\gamma(t)} W^s((u_+,v_+);V), \\
 \mathrm{coker} D_{DV} & \cong X / \big (T_{\gamma(t)} W^u((u_-,v_-);V) + T_{\gamma(t)} W^s((u_+,v_+);V)\big ).
 \end{align*}
 In particular, $D_{DV}$ is onto if and only if $W^u((u_-,v_-);V)$ and $W^s((u_+,v_+);V)$ meet transversally at $\gamma(t)$ for some (and hence all) $t\in \R$. 
\label{cor:surjective}
\end{cor}

\begin{proof}
All claims follow adapting the proof of Propositions 1.6 and 1.8 in \cite{Abbondandolo:2006lk} to our setting, noticing that there the finite dimensionality assumptions are only needed to deduce that the operator is Fredholm. In our case, the Fredholmness is given Corollary \ref{cor:Fredholmness} above. 
\end{proof}

We are now ready to prove the Morse-Smale condition up to order 2 for the functional $f$ as in \eqref{eq:functional} after a generic perturbation of the gradient-like vector field $W$ given by Proposition \ref{prop:gradientc2}. To this purpose, we employ the abstract transversality theorem proved in \cite[Section 5]{Asselle:2022}. We shall notice that Theorem 5.5 in \cite{Asselle:2022} is proved for Hilbert manifolds only, but the extension to the setting of the present paper presents no difficulties as we will discuss below. For this reason, we only give a sketch of the proof, referring to \cite{Asselle:2022} for the details. The Morse-Smale condition up to order 2 will follow from a version of the Sard-Smale theorem due to Quinn and Sard \cite{QS:72} which in the setting of the present paper reads: \textit{let $\varphi:Y\to Z$ be a $C^2$-smooth $\sigma$-proper Fredholm map between the Banach spaces $Y$ and $Z$ with Fredholm index at most 2. Then, the set of regular values of $\varphi$ is generic in $Z$}. Recall that $\varphi$ is called $\sigma$-proper if $Y$ is the countable union of open sets, on the closure of each of which $\varphi$ is proper, and is required as the Banach spaces we are interested in do not satisfy the Lindel\"of property (a necessary 
    condition for the original version of the Sard-Smale theorem).  To explain how we apply the Sard-Smale theorem, let us consider
 neighborhoods $\mathcal U\subset \mathcal V \subset X$ of the set crit$(f)$ of critical points of $f$ such that each critical point of $f$ belongs to a different connected component of $\mathcal V$, and let $\mathfrak C$ be the space of compact vector fields $C$ of class $C^2$ on $X$ having at most linear growth and such that:
 \begin{enumerate}
 \item[(B1)] every $C\in \mathfrak C$ vanishes on $\mathcal U$.
 \end{enumerate}
On $\mathfrak C$ we can introduce a norm $\|\cdot\|_{\mathfrak C}$ which induces the topology of $C^2_{\mathrm{loc}}$-convergence and such that:
\begin{enumerate}
  \item[(B2)] for every $C\in \mathfrak C$ with $\|C\|_{\mathfrak C}\leq 1$, the set of rest points of $W+C$ coincides with crit$(f)$, $f$ is a Lyapounov function for $W+C$, and $(f,W+C)$ satisfies the Palais-Smale condition. 
  \end{enumerate}
For instance, pick a smooth function $\chi:[0,+\infty)\to \R$ such that 
$$0<\chi (\rho) < \frac 12 \inf_{B_\rho(0)\setminus \mathcal U} - \mathrm{d} f[W], \quad \forall \rho \geq 0,$$
where $B_\rho(0)$ denotes the open ball with radius $\rho$ around the origin in $X$, and define for every $C\in \mathfrak C$
$$\|C\|_{\mathfrak C} := \| \chi^{-1}\cdot C\|_{C^2}.$$
The straightforward proof that (B2) is satisfied is left to the reader. Notice that $\mathfrak C$ also satisfies:
\begin{enumerate}
\item[(B4)] $\mathfrak C$ is closed under multiplication by a vector space of functions which includes bump functions,
\item[(B5)] $\{C(u) \, |\, C\in \mathfrak C\} =T_uX\cong X$, for all $u\in X\setminus \mathcal V$.
\item[(B6)] $W+C$ preserves essentially vertical sets (actually, even of the form \eqref{eq:compactperturbation}).
\end{enumerate}
We also notice that Properties (B1), (B2), and (B6), imply using Theorem \ref{teo:precompactness} that  
$$W^u((u_-,v_-);W+C) \cap W^s((u_+,v_+);W+C)$$
is pre-compact for all $\|C\|_{\mathfrak C}\leq 1$ and all $(u_\pm,v_\pm)\in$ crit$(f)$. 

Assume now that $(u_\pm,v_\pm)\in$ crit$(f)$ are such that $\mu_-(u_-,v_-)-\mu_(u_+,v_+)\leq 2$ and 
$$W^u((u_-,v_-);W)\cap W^s((u_+,w_+);W) \neq \emptyset.$$
We define the Banach space
$$K:= C_{\pm}^1(\R,X) := \big \{ \varphi:\R\to X \ |\ \varphi(t) \stackrel{t\to \pm \infty}{\longrightarrow} (u_\pm,v_\pm), \ \dot \varphi(t) \stackrel{t\to \pm \infty}{\longrightarrow} 0 \big \}$$
and observe that the tangent space to $K$ at each $\varphi\in K$ can be identified with 
$$T_\varphi K \cong C^1_{0}(\R,X) \subset B := C^0_0(\R,X).$$
Finally, denoting with $\mathfrak C_1$ the unit ball of $\mathfrak C$ we set 
$$\Phi:\mathfrak C_1 \times K \to B, \quad (C,\varphi)\mapsto \varphi'- (W+C)\circ \varphi,$$
so that 
$$\mathcal Z :=\Phi^{-1}(0)  = \bigcup_{C\in \mathfrak C_1} \Big ( W^u((u_-,v_-);W+C) \cap W^s((u_+,v_+);W+C) \Big ).$$
The fact that $F$ is of class $C^2$ together with the fact that the topology on $\mathfrak C$ coincides with the topology of $C^2_{\mathrm{loc}}$-convergence implies that $\Phi$ is of class $C^2$. The Fredholm theory 
discussed in Lemma \ref{lem:Fredholmness} and Corollaries \ref{cor:Fredholmness} and \ref{cor:surjective} (these serve as a replacement of Lemma 5.6 in \cite{Asselle:2022} which cannot be used in our setting) implies that, for $(C,\varphi) \in \mathcal Z$, $\mathrm d_\varphi \Phi(C,\varphi)$ is Fredholm with Fredholm index $\mu_-(u_-,v_-)-\mu_-(u_+,v_+)$, and it is onto if and only if $W^u((u_-,v_-);W+C)$ and $W^s((u_+,v_+);W+C)$ meet transversally along $\varphi$. This together with Properties (B4) and (B5) implies that 0 is a regular value for $\Phi$, so that $\mathcal Z$ is a $C^2$-submanifold of $\mathfrak C_1\times K$; see Lemma 5.7 in \cite{Asselle:2022} for the details (such a lemma extends verbatim to the Banach setting).  

Let now $\mathcal S \subset \mathcal U$ be a small smooth sphere centered at $(u_-,v_-)$ and transversal to the flow of $W$ (hence, also to the flow of $W+C$ for every $C \in \mathfrak C$ by Property (B1)). We denote by $\mathcal Z_0\subset \mathcal Z$ the codimension-one $C^2$-submanifold given by pairs $(C,\varphi)\in\mathcal Z$ such that $\varphi(0) \in \mathcal S$, and by 
$$\pi :\mathcal Z_0 \to \mathfrak C_1, \quad (C,\varphi) \mapsto C,$$ 
 the projection onto the first factor. One readily sees that $\pi$ is Fredholm of index $\mu_(u_-,v_-)-\mu_-(u_+,v_+)-1$, and that $C\in\mathfrak C_1$ is a regular value of $\pi$ if and only if 
$W^u((u_-,v_-);W+C)$ and $W^s((u_+,v_+);W+C)$  have transverse intersection. Proposition 5.9 and Theorem 5.10 in \cite{Asselle:2022} now extend with no changes to the setting of the present paper (the only ingredients needed in the proofs being the Palais-Smale condition, the fact that all vector fields $W+C$ preserve essentially vertical sets, and the closure of the essentially vertical family with respect to the Hausdorff distance), thus showing that $\pi$ is $\sigma$-proper. 
 
We are now in position to apply the Sard-Smale theorem, thus obtaining that the set $\mathfrak C_1(u_\pm,v_\pm)$ of regular values of the map $\pi$ is generic in $\mathfrak C_1$. Since the set crit$(f)$ is at most countable (actually finite), the intersection 
 $$\mathfrak C_1^{MS}:= \bigcap \Big \{\mathfrak C_1(u_\pm,v_\pm) \ \Big |\ (u_-,v_-)\neq (u_+,v_+)\in \text{crit}(f), \ \mu_-(u_-,v_-)-\mu_-(u_+,v_+)\leq 2\Big \}$$ 
 is also a generic subset of $\mathfrak C_1$, and by construction, for every $C\in \mathfrak C_1^{MS}$, the vector field $W+C$ satisfies the Morse-Smale property up to order 2. 
 
 Once the Morse-Smale property up to order two is achieved (by a generic perturbation of $W$) the construction of Morse homology is standard. We refer to \cite[Section 6]{Asselle:2022} and references therein for the details.
 
 
 \section{Functoriality}
 
 \label{sec:5}
 
 In the previous sections we showed that, associated with a Morse function $f$ as in \eqref{eq:functional} with non-linearity $G$ satisfying  \eqref{eq:growthG} there is a well-defined Morse homology $HM_*(f,W;\Z_2)$, where $W$ is a (generic) gradient-like vector field for $f$ satisfying all properties of Proposition \ref{prop:gradientc2} such that $(f,W)$ is Morse-Smale up to order two. 
 It is now a standard argument (and indeed the proof presented in \cite[Section 7]{Asselle:2022} goes through word by word) to show that different choices of non-linearities $G_1$ and $G_2$ satisfying \eqref{eq:growthG} yield to isomorphic chain complexes. Below we state the functoriality theorem in rigorous way for the reader's convenience. 
 
 \begin{thm}
 The following statements hold:
 \begin{enumerate}
\item Let $f$ as in \eqref{eq:functional}, with non-linearity $G$ satisfying \eqref{eq:growthG}, be a Morse function, and let $W,\tilde W$ be gradient-like vector fields for $f$ satisfying the assumptions in Proposition \ref{prop:gradientc2} and such that $(f,W)$ and $(f,\tilde W)$ are Morse-Smale up to order two. Then, the corresponding Morse complexes are isomorphic. In particular, the induced homology does not depend on the choice of the gradient-like vector field. 
\item Let $f_0,f_1,f_2$ as in \eqref{eq:functional}, with non-linearities $G_0,G_1,G_2$ satisfying \eqref{eq:growthG}, be Morse functions, and assume that $G_0\leq G_1\leq G_2$. Then, for $0\leq i\leq j\leq 2$, there is a sequence of homomorphisms of Abelian groups 
$$\phi_{G_i,G_j}: HM_k(f_i;\Z_2) \to HM_k (f_j;\Z_2),\quad k\in \Z,$$
such that $\phi_{G_1,G_2}\circ \phi_{G_0,G_1} = \phi_{G_0,G_2}$ and $\phi_{G_0,G_0+c}=\mathrm{id}$ for all $c\geq 0$. In particular, the induced homology is independent of the choice of non-linearity satisfying \eqref{eq:growthG}. 
\end{enumerate}
 \label{teo:functoriality}
 \end{thm}

As a corollary, we obtain Theorem \ref{thm:main}.

 \begin{thm}
 Let $f$ in \eqref{eq:functional}, with $G$ satisfying \eqref{eq:growthG}, be a Morse function. Then, Morse homology for $f$ is well-defined and isomorphic to $H_*(X;\Z_2)$. In particular, the system of quasilinear elliptic problems \eqref{eq:system} has at least one solution for every choice of non-linearity satisfying \eqref{eq:growthG}. 
 \label{teo:homologytrivial}
 \end{thm}
 \begin{proof}
 If $f$ is a Morse function, the claim follows from Theorem \ref{teo:functoriality}, as we can equivalently compute $HM_*(f_0,W;Z_2)$, where $f_0$ is associated with $G\equiv 0$. In this case, $f_0$ has a unique critical point, namely $(0,0)\in X$, which is obviously non-degenerate and has relative Morse index $\mu_-(0,0)=0$, and we can choose $W=(-\text{id},\text{id})$ as gradient-like vector field for $f$. As $(0,0)$ is the only critical point, the associated chain complex has only one generator in degree zero and the boundary operator vanishes identically, thus yielding 
 $$HM_*(f_0,W;Z_2) = \left \{ \begin{array}{r} \mathbb Z_2 \ \quad \text{for} \ *=0, \\ (0) \quad \text{otherwise}.\end{array}\right.$$
 If $f$ is not a Morse function, then there must exist at least one degenerate critical point. 
 \end{proof}
 
Abstracting from the concrete setting of Theorem \ref{teo:homologytrivial}, we now provide sufficient conditions ensuring that the Morse homology of a $C^2$-smooth strongly indefinite functional 
$f$ on a Banach space is well defined. 
 
 \begin{thm}
 Let $X$ be a reflexive Banach space, and let $f:X\to \R$ be a functional of class $C^2$ satisfying the Palais-Smale condition. Assume further that there exists a splitting 
 $X= \mathcal E^- \oplus \mathcal E^+$
 such that:
 \begin{enumerate}
 \item[1.] $\mathrm{(compatible\ splitting\ at\ critical\ points)}$ For each $\bar x \in \mathrm{crit} (f)$, there exist a splitting $X = X^- \oplus X^+$ and a constant $c>0$ such that 
 $$\pm \mathrm d^2f(\bar x)[v^\pm,v^\pm] \geq 2c \|v^\pm \|^2_w, \quad \forall v\in X^\pm,$$
 for some possibly weaker norm $\|\cdot\|_w$, and such that $(\mathcal E^-, X^+)$ is a Fredholm pair. 
 \item[2.] $\mathrm{(persistence\ of\ splitting)}$ For each $\bar x \in \mathrm{crit} (f)$ there exists $\delta >0$ such that 
 $$\pm \mathrm d^2f(\bar x+y)[v^\pm,v^\pm] \geq c \|v^\pm \|^2_w, \quad \forall v\in X^\pm, \ \forall \|y\|<\delta.$$
 \item[3.] $\mathrm{(gradient}$-$\mathrm{like\ vector\ field)}$ The local gradient-like vector fields $L_{\bar x}$ for $f$ 
 $$L_{\bar x} = \mathrm{id}_{X^-} - \mathrm{id}_{X^+} = \mathrm{id}_{\mathcal E^-} - \mathrm{id}_{\mathcal E^+} + K_{\bar x},$$
given by Proposition \ref{prop:hyperbolic}, can be extended to a $C^2$-smooth gradient-like vector field $W$, with linear growth and such that $(f,W)$ satisfies the Palais-Smale condition,  of the form 
 $$W= \mathrm{id}_{\mathcal E^-} - \mathrm{id}_{\mathcal E^+} + K$$
 with $K:X\to X$ compact.
 \end{enumerate}
 Then, the Morse homology $HM_*(f;\Z_2)$ is well-defined. 
 \label{thm:abstract}
 \end{thm}
 
We conclude the paper with a brief discussion of how far we expect the theorem to generalize to broader settings.
First, Conditions 1 and 2 imply that critical points are non-degenerate in the sense of Definition \ref{def:nondegintro}. Following \cite{Asselle:2025}, Condition 2 may be relaxed 
 by requiring that its failure be suitably controlled. 
Second, Condition 3 implies that stable and unstable manifold of pairs of critical points have precompact intersection. As in Section \ref{sec:4}, these intersections can be made transverse by a generic perturbation of the gradient-like vector field $W$ within the class of vector fields of the form ``compact perturbation of $\mathrm{id}_{\mathcal E^-} - \mathrm{id}_{\mathcal E^+}$''. Following \cite{AM:03},  the Fredholm analysis of the operator $D_A$ in \eqref{eq:DA} is expected to generalize to continuous paths $A(t) = \tilde A(t) + K(t),$
 where $\tilde A(t)$ respects the splitting $X=\mathcal E^-\oplus \mathcal E^+$ and $K(t)$ is compact for all $t\in \R$. According to \cite{AM:03}, this should be the most general situation in which the good Fredholm properties of $D_A$ persist. We do not pursue this extension here, and leave it to future work instead.


\bibliography{_biblio}
\bibliographystyle{plain}

\end{document}